\documentclass{l4dc2025}

\title[Flow matching for  control systems]{Flow matching for stochastic linear control systems}
\usepackage{times}
\usepackage{url}

\author{%
 \Name{Yuhang Mei} \Email{yuhangm@uw.edu}\\
 \addr Department of Aeronautics and Astronautics, University of Washington, Seattle, WA
 \AND
 \Name{Mohammad Al-Jarrah} \Email{mohd9485@uw.edu}\\
 \addr Department of Aeronautics and Astronautics, University of Washington, Seattle, WA
  \AND
 \Name{Amirhossein Taghvaei} \Email{amirtag@uw.edu}\\
 \addr Department of Aeronautics and Astronautics, University of Washington, Seattle, WA
 \AND
  \Name{Yongxin Chen} \Email{ychen3148@gatech.edu}\\
 \addr School of Aerospace Engineering, Georgia Institute of Technology, Atlanta, GA%
}

\newcommand{\ud}{\,\mathrm{d}}

\newcommand{\trace}{{Tr}}

\renewcommand{\tilde}{\widetilde}

\renewcommand{\Re}{\mathbb{R}}

\newcommand{\argmin}{\mathop{\operatorname{argmin}}}

\newcommand{\Pinit}{P_{\text{initial}}}
\newtheorem{assumption}{Assumption}
\newtheorem{problem}{Problem}

\begin{document}

\maketitle

\begin{abstract}%
This paper addresses the problem of steering an initial probability distribution to a target probability distribution through a deterministic or stochastic linear control system. Our proposed approach is inspired by the flow matching methodology, with the difference that we can only affect the flow through the given control channels. The motivation comes from applications such as robotic swarms and stochastic thermodynamics, where agents or particles can only be manipulated through control actions. The feedback control law that achieves the task  is characterized as the conditional expectation of the control inputs for the stochastic bridges that respect the given control system dynamics. Explicit forms are derived for special cases, and a numerical procedure is presented to approximate the control law, illustrated with  examples. 
\end{abstract}

\begin{keywords}%
Flow matching,  Stochastic control, Mean-field control
\end{keywords}

\section{Introduction}

Flow matching has recently gained attention as a promising method for generative modeling due to its simplicity and flexibility~\citep{lipman2022flow,liu2022flow,albergo2022building,tong2023improving}. From a control-theoretic perspective, the methodology can be understood as follows. Consider the control system:
\begin{equation}\label{eq:simple-sys}
\frac{\ud  X_t}{\ud t} = u_t,\quad X_0 \sim \Pinit,  
\end{equation}
where $\{X_t\in \Re^n;0\leq t\leq 1\}$ is the state, $\{u_t\in \Re^n;0\leq t\leq 1\}$ is the control input, and $ \Pinit$ is the distribution of the initial state $X_0$. The control objective is to find a control input $u_t$ such that the terminal state $X_1$ follows a desired target distribution $P_{\text{target}}$. Flow matching offers a straightforward solution.  First, a probability flow $\{P_t;0\leq t\leq 1\}$ is constructed on the space of probability distributions. This flow is chosen to interpolate between the initial and target distributions, i.e. $P_0= \Pinit$ and $P_1=P_{\text{target}}$, and is easy to sample from. A standard choice for $P_t$ is the probability law of the linear interpolation process $X^z_t=(1-t)x + ty$ where $z=(x,y)\sim  \Pi:=\Pinit \otimes P_{\text{target}}$.  Then, the control input $u_t$ is identified so that the probability of $X_t$, and $P_t$, both satisfy the same continuity equation. The resulting control input takes the form $u_t=k(t,X_t)$ where the feedback control law $k:[0,1]\times \Re^n \to \Re^n $ has the probabilistic representation \[k(t,\xi)=\mathbb E[\frac{\ud X^z_t}{\ud t} |X^z_t=\xi],\quad \forall (t,\xi) \in [0,1]\times \Re^n.\] Through this procedure, the probability law of $X_t$ matches $P_t$, for all $t\in[0,1]$, achieving the control objective $X_1\sim P_1=P_{\text{target}}$. 
A key computational advantage of flow matching is that the feedback control law $k(t,\cdot)$ can be numerically approximated by solving a least-squares regression problem:
\begin{equation*}
    \min_k\,\mathbb E_{z \sim \Pi}[\|k(t,X^z_t) - \frac{\ud X^z_t}{\ud t} \|^2].
\end{equation*}

The aim of this paper is to extend the flow matching methodology to the general control setting where the simple control system~\eqref{eq:simple-sys} is replaced by a general deterministic or stochastic linear control system of the form~\eqref{eq:lin-dyn-det} or~\eqref{eq:lin-dyn-stoch}.  The notable difference from traditional flow matching is that here, adjustments to the differential equation are limited to control inputs, a constraint arising from engineering applications such as robotic swarms~\citep{chen2020mean,9735297,elamvazhuthi2019mean,liu2018mean} or stochastic thermodynamic systems~\citep{sekimoto2010stochastic,peliti2021stochastic,seifert2012stochastic,chen2019stochastic,fu2021maximal,movilla2023}, where agents or particles can only be manipulated through control actions.

The problem of controlling probability distributions has a rich history in control theory, dating back to Roger Brockett’s work on the control of Liouville  equations~\citep{brockett2007optimal,brockett2012notes}. Interest in this area has expanded due to its connections with mean-field games~\citep{huang2006large,lasry2007mean,chen2018steering}, mean-field control~\citep{bensoussan2013mean,carmona2018probabilistic,fornasier2014mean}, optimal transportation/Schr\"odinger bridge problem~\citep{chen2016optimal,chen2016relation,9491012,haasler2021control,zhou2021optimal,chen2023density}. 

Namely, our work is closely related to~\cite{chen2015optimal} which derives the optimal feedback control law that steers a stochastic linear control system from an initial Gaussian distribution to a Gaussian target distribution in an optimal manner. The flow matching methodology presented here generalizes the framework to non-Gaussian distributions,  though it no longer guarantees optimality. Our work is also closely related to~\cite{liu2023generalized}  where flow matching  is used to solve the generalized schr\"odinger bridge problem in an alternating optimization scheme.  The difference in our setup is constraining the dynamics to linear control systems of the form~\eqref{eq:lin-dyn-det} or~\eqref{eq:lin-dyn-stoch} and forgoing optimality.  While some notion of optimality could be introduced by designing an optimal coupling between the initial and target distributions (e.g., using the Sinkhorn algorithm for optimal sample pairing), this is not the focus of our work.

This paper is organized as follows. Section~\ref{sec:interpolation} presents interpolations over deterministic and stochastic linear control system.  Section~\ref{sec:flow-mathcing} presents the generalization of the flow matching methodology to stochastic linear control systems, followed by the analytical derivation of the control law for special cases of Gaussian and mixture of Gaussian target distribution. Finally, Section~\ref{sec:numerics} presents a numerical procedure which is demonstrated with the aid of several examples.

\section{Background on interpolation through linear control systems}\label{sec:interpolation}

 In this section, we present interpolations that satisfy a given deterministic or stochastic linear control system. 
  \subsection{Deterministic linear control system} 
  Consider the linear control system
  \begin{align}\label{eq:lin-dyn-det}
  	\frac{\ud X_t}{\ud t} = AX_t + Bu_t,
  \end{align}
where $X_t \in \mathbb R^n$ is the state and $u_t \in \mathbb R^m$ is the control input, at time $t$. We consider the following control problem.  

\begin{problem} \label{problem:cont-det} Given a pair of points $(x,y) \in \mathbb R^n \times \mathbb R^n$, find a trajectory $\{X_t;t\in[0,1]\}$ such that  $X_0=x$, $X_1=y$, and~\eqref{eq:lin-dyn-det} is satisfied for some control input $\{u_t;t\in[0,1]\}$.  
\end{problem}
This is a standard problem in control theory, forming the basis for controllability analysis of linear systems, e.g. see~\cite[Ch. 5]{basar2020lecture}.  
In order to solve this problem, it is useful to define the controllability Gramian 
	\begin{align*}
		\Phi_t := \int_0^te^{(t-s)A}BB^\top e^{(t-s)A^\top}\ud s,\quad \text{for}\quad t\in[0,1],
	\end{align*}
and make the following assumption about the system. 

\begin{assumption}\label{assumption:controllable}
    The pair $(A,B)$ is controllable. That is to say, the matrix $[B,AB,A^2B,\ldots,A^{n-1}B]$ is full-rank. 
\end{assumption}
Under the controllability Assumption~\ref{assumption:controllable}, it is known that the Gramian matrix $\Phi_t$ is non-singular for all $t > 0$~\citep[Th. 5.1]{basar2020lecture}.  
The following proposition states a solution to Problem~\ref{problem:cont-det}. The proof is standard and omitted on account of space.

\begin{proposition}
Problem \ref{problem:cont-det} is solved with the control input 
\begin{align}\label{eq:u-det}
		u_t &= B^\top e^{(1-t)A^\top}\Phi_1^{-1}(y-e^Ax),
	\end{align} 
    resulting into the interpolating trajectory 
	\begin{align}\label{eq:det-interpolation}
		X_t =  e^{tA}x + \Phi_t  e^{(1-t)A^\top}\Phi_1^{-1}(y-e^Ax),\quad \text{for}\quad t\in[0,1]. 
	\end{align}  
    Moreover,~\eqref{eq:u-det} is the control input with minimum $L_2$-norm  $\int_0^1\|u_t\|^2\ud t$ among all the control inputs that solve Problem~\ref{problem:cont-det}. 
\end{proposition}
\begin{remark}
    It is useful to express the control input~\eqref{eq:u-det} as a feedback control law $u_t=k(t,X_t)$ where 
    \begin{align}\label{eq:u-det-feedback}
        k(t,\xi) = B^\top e^{(1-t)A^\top} \Phi_{1-t}^{-1} (y-e^{(1-t)A}\xi),\quad \forall (t,\xi) \in [0,1]\times \Re^n. 
    \end{align}
    This is obtained by using~\eqref{eq:det-interpolation} to solve for $x$, in terms of $X_t$ and $y$, and substituting the result in~\eqref{eq:u-det}. The feedback control law steers the system from any given initial point to $y$. Surprisingly, the same feedback control law achieves this task in the stochastic setting. 
\end{remark}
\begin{remark}
The interpolation formula can be generalized to a linear time-varying system $\frac{\ud X_t}{\ud t} = A_tX_t + B_tu_t$ by replacing $e^{(t-s)A}$ by the corresponding state transition matrix $\Psi_{t,s}$ where $\frac{\ud \Psi_{t,s}}{\ud t} = A_t \Psi_{t,s}$, $\Psi_{s,s}=I$, for all $t\geq s\geq 0$, and replacing $B$ by its time-varying counterpart $B_t$.   
\end{remark}

  \subsection{Stochastic linear control system} 
   Consider the stochastic linear control system:
\begin{align}\label{eq:lin-dyn-stoch}
  	\ud X_t = AX_t \ud t+ B(u_t\ud t + \epsilon\ud W_t),
  \end{align}
where $\{W_t\}_{t\geq 0}$ is  $n$-dimensional Brownian motion and $\epsilon\in \mathbb R$. Let $\mathcal F_t:=\sigma(W_s;0\leq s\leq t)$ denote the filtration generated by the Brownian motion. 

\begin{problem} \label{problem:cont-stoch} For any pair $z=(x,y)\in \mathbb R^n \times \mathbb R^n$, find a stochastic trajectory $\{X_t;t\in[0,1]\}$ such that  $X_0=x$, $X_1=y$, and~\eqref{eq:lin-dyn-det} is satisfied for some control input $\{u_t;t\in[0,1]\}$ that is $\mathcal F_t$-adapted.
\end{problem}
Problem~\ref{problem:cont-stoch} can be solved by constructing stochastic bridges, that is to say, the uncontrolled process $X_t$ conditioned at its end-points, $X_0=x$ and $X_1=y$. These stochastic bridges have been developed for general non-degenerate diffusions~\citep{dai1991stochastic,pavon1989stochastic,pavon1991free} and for degenerate diffusion of the type~\eqref{eq:lin-dyn-stoch} in~\cite{chen2015stochastic}, using a stochastic optimal control formulation of the problem~\ref{problem:cont-stoch}. Here, we present an alternative approach using the time-reversal methodology~\citep{anderson1982reverse,haussmann1986time} and derive the formula for the feedback control law that solves problem~\ref{problem:cont-stoch}.  We present this approach for its simplicity and flexibility, while it should be noted that the results are the same as in~\cite{chen2015stochastic}. 
\begin{proposition}
 Problem~\ref{problem:cont-stoch} is solved with the feedback control law 
 \begin{align}\label{eq:u-stoch}
	u_t &= B^\top e^{(1-t)A^\top} \Phi_{1-t}^{-1} (y-e^{(1-t)A}X_t). 
\end{align}
The marginal probability of the resulting trajectory  is 
	 \begin{align}\label{eq:stoch-interpolation}
	 	X_t \overset{\text{d}}{=}  e^{tA}x + \Phi_t  e^{(1-t)A^\top}\Phi_1^{-1}(y-e^Ax) + \epsilon \Sigma_t^{\frac{1}{2}}Z,
	 \end{align} 
 where $Z$ is normal Gaussian, $\overset{\text{d}}{=}$ means equality in distribution, and
 \begin{equation*}
 	\Sigma_t := \Phi_t - \Phi_t e^{(1-t)A^\top} \Phi_1^{-1}e^{(1-t)A} \Phi_t.
 \end{equation*}
\end{proposition}
\begin{remark}
	The feedback control law that achieves the deterministic and stochastic interpolation  is exactly the same: 
	\begin{equation*}
		k(t,\xi) = B^\top e^{(1-t)A^\top} \Phi_{1-t}^{-1} (y-e^{(1-t)A}\xi).
	\end{equation*}
    The expectation of the stochastic interpolation~\eqref{eq:stoch-interpolation} is exactly equal to the deterministic interpolation~\eqref{eq:det-interpolation}. Moreover, the deterministic interpolation can be derived from the stochastic interpolation in the limit as $\epsilon \to 0$. 
\end{remark}
\begin{proof}
	The uncontrolled process $X_t$, starting from $X_0=x$, is Gaussian with mean and covariance matrix given by:
	\begin{align*}
		\mathbb E[X_t]=e^{tA}x,\quad \text{Cov}(X_s,X_t) = \epsilon^2\Phi_s e^{(t-s)A^\top},\quad\text{for}\quad t\geq s\geq 0.
	\end{align*}
Therefore, conditioning on $X_1=y$, $X_t$ remains Gaussian with the mean  and covariance 
\begin{align*}
	\mathbb E[X_t|X_1=y] &= \mathbb E[X_t] + \text{Cov}(X_t,X_1)\text{Cov}(X_1,X_1)^{-1}(y-\mathbb E[X_1]), \\
	\text{Cov}(X_t,X_t|X_1=y) &= \text{Cov}(X_t,X_t)  - \text{Cov}(X_t,X_1) \text{Cov}(X_1,X_1)^{-1} \text{Cov}(X_t,X_1)^\top, 
\end{align*}
concluding the formula~\eqref{eq:stoch-interpolation}.

In order to obtain the formula for the control input, 
consider the uncontrolled process $\tilde X_t$ satisfying 
\begin{align*}
	\ud \tilde X_t = - A \tilde X_t  \ud t + \epsilon B \ud \tilde W_t,\quad \tilde X_0=y.
\end{align*}
The probability distribution of $\tilde X_t$ is Gaussian $\mathcal N(m_t,Q_t)$ with mean $m_t = e^{-At}y$ and covariance matrix \[Q_t = \epsilon^2\int_0^t e^{-sA} BB^\top  e^{-sA^\top} = \epsilon^2 e^{-tA} \Phi_t e^{-tA^\top}.\]
The time-reversal $X_t:=\tilde X_{1-t}$ satisfies the SDE
\begin{align*}
	\ud X_t = A X_t\ud t + \epsilon B \ud W_t -\epsilon^2 B B^\top Q_{1-t}^{-1}(X_t  - m_{1-t})\ud t.
\end{align*}
Note that, by construction, $X_1 = \tilde X_{0} = y$.  This is true starting from any initial point $X_0=x$. Therefore the control law that achieves the stochastic interpolant is 
\begin{align*}
	u_t = - \epsilon^2 B^\top Q_{1-t}^{-1}(X_t - m_{1-t}) =- B^\top e^{(1-t)A^\top}\Phi_{1-t}^{-1}e^{(1-t)A}(X_t - e^{-(1-t)A}y),
\end{align*}
concluding the control law~\eqref{eq:u-stoch}.
\end{proof}
\begin{remark} Our results hold for the case where we replace $\epsilon$ with a matrix. However, 
we restrict the exposition to stochastic models where the Brownian motion enters the system from the same channels as the control input.   
\end{remark}

\begin{figure}[t]
    \centering
    \includegraphics[width=0.32\textwidth,trim={10 0 45 40},clip]{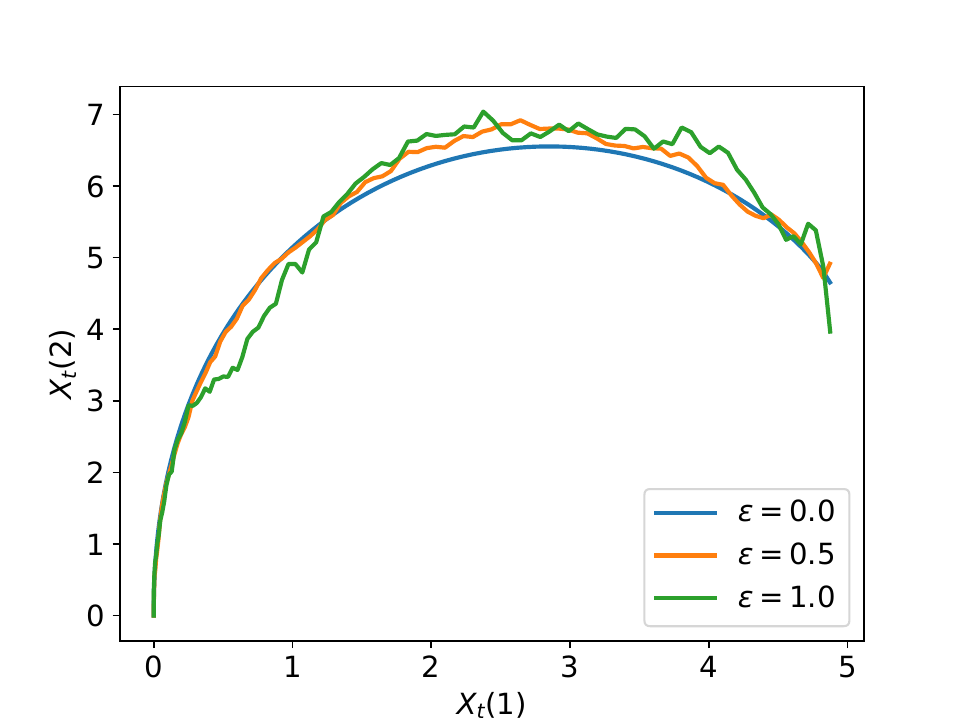}
         \hspace{1pt}
         \includegraphics[width=0.32\textwidth,trim={10 0 45 40},clip]{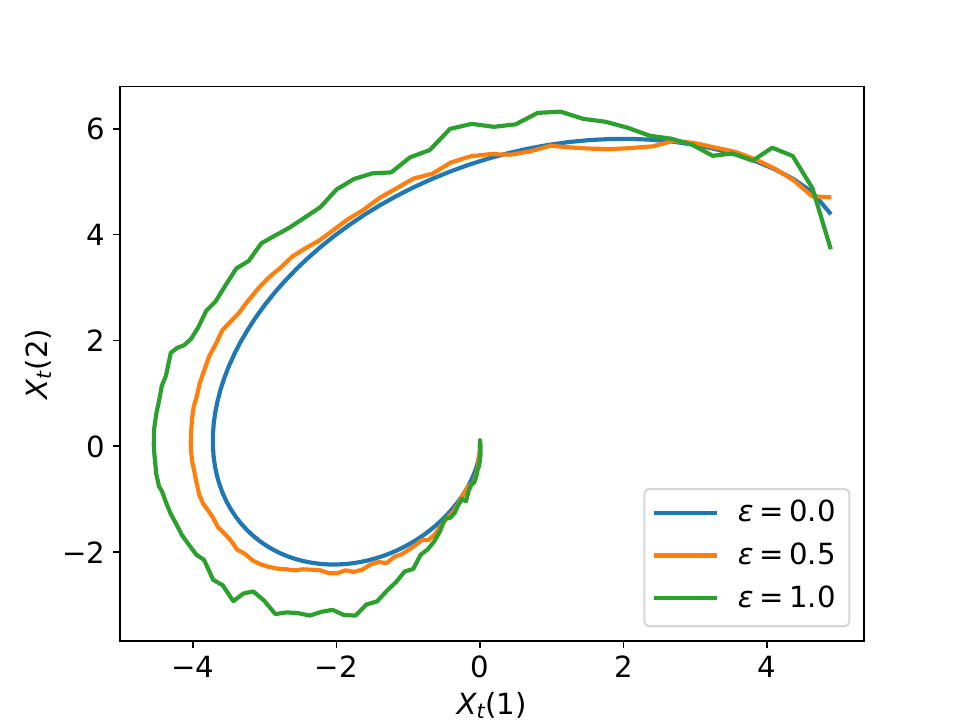}
         \hspace{1pt}
         \includegraphics[width=0.32\textwidth,trim={10 0 45 40},clip]{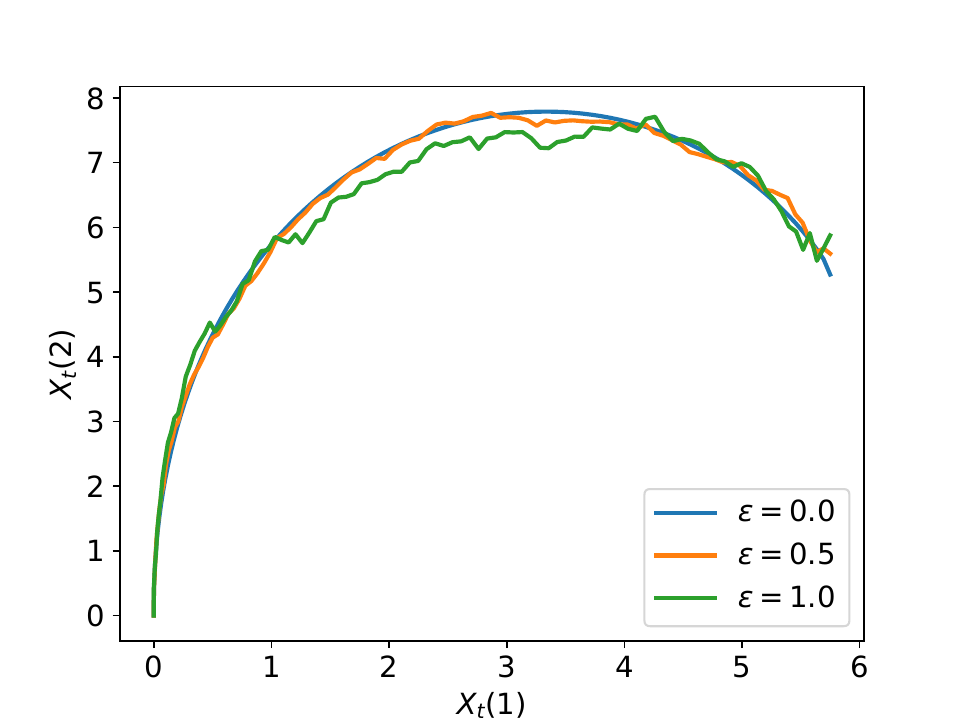} \\
         \hspace{10pt}(a) \hspace{130pt} (b)  \hspace{130pt} (c)
    \caption{Interpolations for three linear control systems with modeling parameters specified in~\eqref{eq:2d_systems}. 
    }
    \label{fig:three system traj}
\end{figure}

Figure~\ref{fig:three system traj} shows interpolations for the following second-order systems,

\begin{align}\label{eq:2d_systems}
    &(a) ~A = \begin{bmatrix}
        0 & 1\\ 0 & 0
    \end{bmatrix} ,\quad 
    (b) ~A = \begin{bmatrix}
        0 & 5\\ -5 & 0
    \end{bmatrix} , \quad
    (c) ~A = \begin{bmatrix}
        0 & 1\\ -1 & -1
    \end{bmatrix}, 
\end{align}
with the same $B = [0 \quad 1]^\top$. The first system represents a double integrator, the second one is a pure oscillator, and the third one is a Nyquist Johnson resistor model.  These systems are later used in our numerical results in Section~\ref{sec:numerics}. 

\section{Flow matching for control systems} \label{sec:flow-mathcing}
Our goal is to use the deterministic and stochastic interpolations constructed in Section~\ref{sec:interpolation} to find the feedback control law that steers a given initial distribution $P_0$ to a desired target distribution $P_1$ through a a deterministic or stochastic linear system. We present the construction for the stochastic case, since the construction for the deterministic case is identical. 

%
\begin{problem}\label{problem:cont-dist} Consider the stochastic linear system~\eqref{eq:lin-dyn-stoch}. For any pair of distributions $(P_0,P_1)$, find a $\mathcal F_t$-adapted control input $u_t$ so that if $X_0 \sim P_0$, then $X_1\sim P_1$. 
\end{problem}
We solve the problem using the flow matching methodology. The first step is to construct a process, denoted by $X^z_t$, that has the desired end-point distributions, i.e. $X^z_0\sim P_0$ and $X^z_1 \sim P_1$. The second step is to find the control input $u_t$ in~\eqref{eq:lin-dyn-stoch}  so that the probability law of $X_t$ is equal to the probability law of $X^z_t$.

In order to achive the first step, let $z=(x,y) \in \Re^n \times \Re^n$ with the  probability distribution $\Pi$. Assume the $x$-marginal and the $y$-marginal of $\Pi$ are equal to $P_0$ and $P_1$, respectively (e.g. take $\Pi=P_0\otimes P_1$). 
For each $z\sim \Pi$, let $X^z_t$ be the stochastic interpolant ~\eqref{eq:stoch-interpolation} with stochastic control input $u^z_t$ given by~\eqref{eq:u-stoch}. Note that, by construction,
\begin{align*}
    X^z_0 = x \sim P_0,\quad \text{and}\quad X^z_1 = y \sim P_1. 
\end{align*}
This achieves the goal of the first step. The second step is presented in the following theorem. 
 \begin{theorem}\label{thm:control}
 	Problem~\ref{problem:cont-dist} is solved with the feedback control law $u_t = \bar k(t,X_t)$ where 
 	\begin{align}\label{eq:u-flow-mathcing}
 		\bar k(t,\xi)&= \mathbb E[u^z_t|X^z_t=\xi]. 
 	\end{align}
    Moreover, $X_t \overset{\text{d}}{=}X^z_t$ for all $t\in[0,1]$, where $\overset{\text{d}}{=}$ means equality in distribution. 
 \end{theorem}

\begin{proof}
		Consider a smooth and bounded test function $f:\mathbb R^n \to \mathbb R$. Let $X_t$ be the solution to~\eqref{eq:lin-dyn-stoch} with the feedback control law~\eqref{eq:u-flow-mathcing}. Then, the It\^o rule implies 
	\begin{align*}
		\frac{\ud}{\ud t} \mathbb E[f(X_t)] = \mathbb E[\nabla f(X_t)(A X_t + B \bar k(t,X_t)) + \epsilon^2\trace(BB^\top\nabla^2 f(X_t)]
	\end{align*}
	On the other hand, taking the time-derivative of  $\mathbb E[f(X^z_t)]$, while using the fact that, for a fixed $z$, $X^z_t$ solves ~\eqref{eq:lin-dyn-stoch} with control input $u^z_t$, implies 
	\begin{align*}
		\frac{\ud}{\ud t} \mathbb E[f(X^z_t)] &= \mathbb E[\nabla f( X^z_t)(AX^z_t + B u^z_t)+\sigma^2\trace(BB^\top\nabla^2 f(X^z_t)] \\&=\mathbb E[\nabla f( X^z_t)(A X^z_t + B\mathbb E [u^z_t|X^z_t])+\sigma^2\trace(BB^\top\nabla^2 f(X^z_t)]\\&=\mathbb E[\nabla f( X^z_t)(A X^z_t + B\bar k(t,X^z_t))+\sigma^2\trace(BB^\top\nabla^2 f(X^z_t)],
	\end{align*}
	where we used the tower property of conditional expectation in the second identity, and the definition of the control law~\eqref{eq:u-flow-mathcing} in the third identity. The two derivations imply that the probability law of $X_t$ and $X^z_t$ follow  the same update law. 
	Therefore, due the the equality of the initial  distribution of $X^z_0\sim P_0$ and $X_0\sim P_0$, the probability law of $X_t$ is equal to the probability law of $X^z_t$ for all $t\in[0,1]$. In particular, probability distribution of $X_1$ is equal to $P_1$, the probability distribution of $X^z_1$. This concludes the solution to Problem~\ref{problem:cont-dist}.

\end{proof}

\subsection{Control law for Gaussian and mixture of Gaussians target distribution}
In this section, we present the analytical formula for the feedback control law for the special case where $P_0$ is a Gaussian and $P_1$ is a mixture of Gaussians. In particular, we assume 
\begin{align}\label{eq:mixture-model}
    P_0 = \eta_0,\quad P_1=\sum_{l=1}^L w_l \eta_l,
\end{align}
where $\eta_l$ is $\mathcal N(m_l,Q_l)$, for $l=0,\ldots,L$, $w_l\geq 0$, and $\sum_{l=1}^L w_l=1$. The case $L=1$ corresponds to a Gaussian target, while $L>1$ corresponds to a mixture. 

 In order to achieve this, it is useful to express the feedback control law~\eqref{eq:u-flow-mathcing} according to
\begin{align}\label{eq:kbar-yhat}
 		\bar k(t,\xi)
 		&=B^\top e^{(1-t)A^\top} \Phi_{1-t}^{-1} (\mathbb E(y|X^z_t=\xi]-e^{(1-t)A}\xi)
 \end{align}
where the formula~\eqref{eq:u-stoch} for $u^z_t$  is used. Therefore, the problem of finding the feedback control law is reduced to finding the analytical formula for the conditional expectation $\mathbb E[y|X^z_t=\xi]$. To simplify the presentation, we express 
    the relationship~\eqref{eq:stoch-interpolation} according to
    \begin{align}\label{eq:Xt-y-x-R-S}
       X^z_t = R_t x + S_ty + \epsilon \Sigma_t^{\frac{1}{2}}Z.
     \end{align}
     where 
          \begin{align*}
         R_t:= e^{tA}-\Phi_t e^{(1-t)A^\top}\Phi_1^{-1}e^A ,\quad \text{and}\quad S_t := \Phi_t e^{(1-t)A^\top}\Phi_1^{-1}.
     \end{align*}
\begin{corollary}\label{cor:mixture}
     Consider Problem~\ref{problem:cont-dist} where $P_0$ is a Gaussian and $P_1$ is a mixture of Gaussians, as described in~\eqref{eq:mixture-model}. Then, the feedback control law that solves the problem takes the form~\eqref{eq:kbar-yhat}
     where 
                 \begin{align}\label{eq:y-hat-mixture}
        \mathbb E[y|X^z_t=\xi]
        &= \frac{1}{\sum_{l=1}^L w'_l}\sum_{l=1}^L w'_l(m_l + K_l(\xi-R_tm_0-S_tm_l)), 
     \end{align}
     and 
     \begin{align*}
         K_l&:=Q_lS_t^\top(R_tQ_0R_t^\top + S_t Q_lS_t^\top + \epsilon^2 \Sigma_t)^{-1},\\
         w'_l&:= w_l\exp(-\frac{1}{2}(\xi-R_tm_0-S_tm_l)^\top(R_t^\top Q_0 R_t+S_tQ_lS_t^\top+\epsilon^2\Sigma_t)^{-1}(\xi-R_tm_0-S_tm_l)),
     \end{align*}
     for $l=1,\ldots,L$. 
\end{corollary}
\begin{proof}
Consider the spacial case where $L=1$. In this case, 
 the Gaussian assumption  $x\sim \mathcal N(m_0,Q_0)$ and $y\sim \mathcal N(m_1,Q_1)$, the relationship~\eqref{eq:Xt-y-x-R-S}, and selecting the independent coupling $\Pi=P_0\otimes P_1$, imply 
     \begin{align*}
        \mathbb E[y|X^z_t=\xi] &=  \mathbb E[y] + \text{Cov}(y,X^z_t)\text{Cov}(X^z_t,X^z_t)^{-1}(\xi-\mathbb E[X^z_t])\\
        &=m_1 + Q_1S_t^\top(R_tQ_0R_t^\top + S_t Q_1S_t^\top + \epsilon^2 \Sigma_t)^{-1}(\xi-R_tm_0-S_tm_1),
     \end{align*}
     concluding the formula~\eqref{eq:y-hat-mixture} for $L=1$. 
     
The extension to the mixture case $L>1$ follows by computing the conditional expectation for each Gaussian member of the mixture, and forming their weighted linear combination, where the weights are appropriately adjusted according to the likelihood of that particular member. The derivation details are removed on account of space. Similar derivations are common in Gaussian sum filters~\citep{alspach1972nonlinear}.   A similar form of the feedback control law also appears in~\cite{salhab2019collective} in the context of LQG games with multiple choice. 
     \end{proof}

\subsection{Extension to control affine systems}
Consider the control affine stochastic system 
\begin{align*}
    \ud X_t = a(X_t)\ud t + b(X_t)(u_t + \epsilon \ud W_t),
\end{align*}
where $a:\Re^n\to\Re^n$ and $b:\Re^n\to \Re^{n\times m}$.  Let $X^z_t$ be the stochastic bridge for this system, with the corresponding control input $u^z_t$. Then, the control law~\eqref{eq:u-flow-mathcing} can also be used for this system, to steer an initial distribution $P_0$ to a target distribution $P_1$. However, in such a general setting, the process of sampling from $X^z_t$ becomes computationally challenging because  an analytical expression for the stochastic bridge is not available. 
\section{Algorithm and numerical results}\label{sec:numerics}
In general, the conditional expectation~\eqref{eq:u-flow-mathcing} does not have an explicit solution. We follow the same procedure as in the flow matching methodology to numerically approximate the conditional expectation as the solution to the least-squares  regression problem: 
 \begin{align}\label{eq:numerical_approx}
 	\bar k &= \argmin_f \mathbb E[\|f(t,X^z_t) - u_t^z\|^2 ]\approx \argmin_{f\in \mathcal F} \frac{1}{N}\sum_{i=1}^N \|f(t,X^{z^i}_t)-u_t^{z^i}\|^2.
 	\end{align}
     where the expectation is approximated using $N$ independent samples of the pair $z^i=(x^i,y^i)\sim \Pi$. In this expression, $X^{z^i}_t$  represents a sample from~\eqref{eq:stoch-interpolation} with the corresponding control input $u_t^{z^i}$ given in~\eqref{eq:u-stoch}.  
And $\mathcal F$ represents a parameterized function class. 
Across all numerical results, we select $\mathcal F$ to be the class of neural networks with a 3 block ResNet architecture where each block consists of 2 linear layers of width $32$ and an exponential linear unit (ELU)-type activation function. 

The algorithm has two stages. In the training stage, we use the ADAM optimizer, with initial learning rate $10^{-2}$ and exponential decay of order $0.999$, to solve the optimization problem~\eqref{eq:numerical_approx} and find the parameters of the neural net $f$. The number of iterations is $10^4$, with the total sample size $N=2000$, and the batch size $64$. In the prediction stage, we use the Euler-Maruyama method, with $\Delta t = 0.001$,  to simulate $N'=2000$ independent realizations of the SDE~\eqref{eq:lin-dyn-stoch}  using the feedback control law learned in the training stage.  The samples are denoted by $\{X^i_t\}_{i=1}^{N'}$. 
The code for reproducing the results is available online\footnote{\url{https://github.com/YuhangMeiUW/Flow-matching-control}}.

     

 \begin{figure}[t]
    \centering
    \includegraphics[width=0.32\textwidth,trim={5 5 55 35},clip]{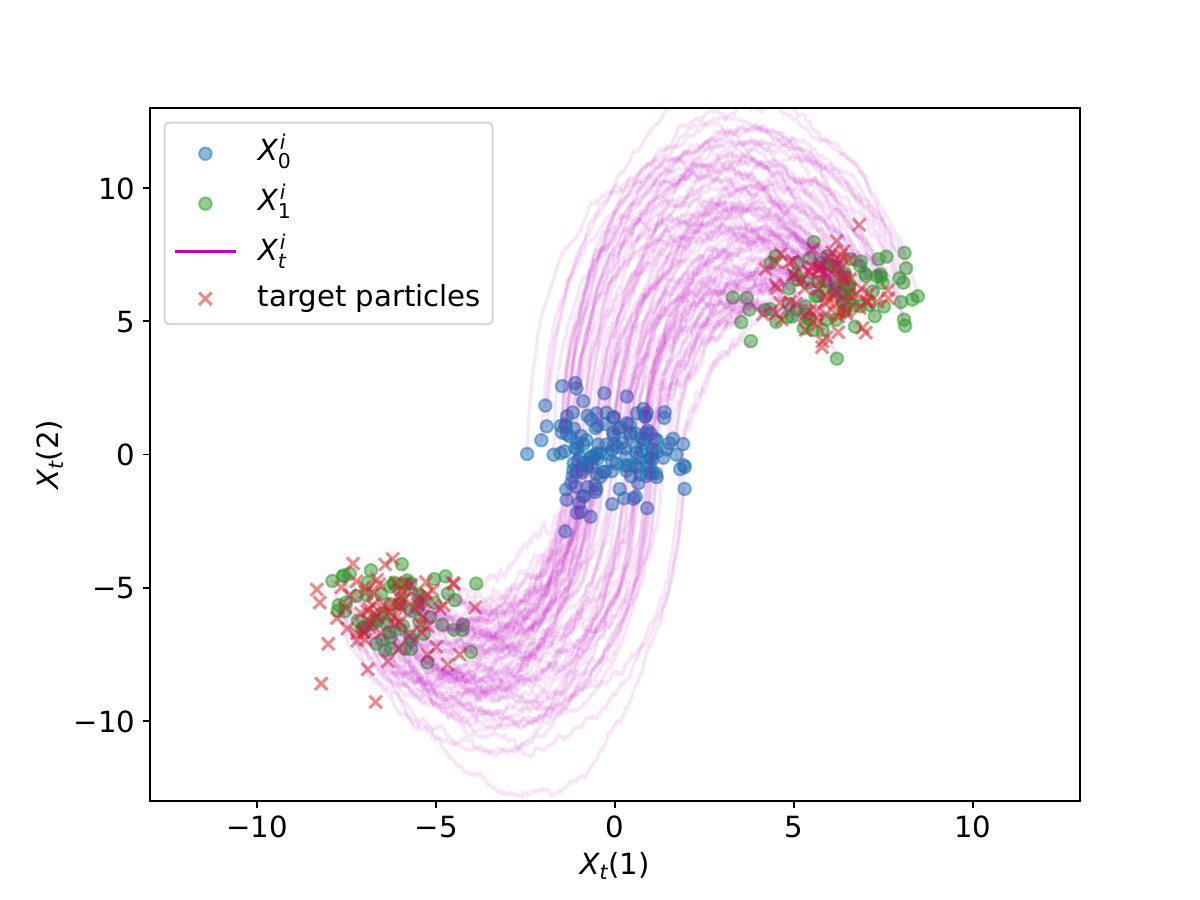}
         \hspace{1pt}
         \includegraphics[width=0.32\textwidth,trim={5 5 55 35},clip]{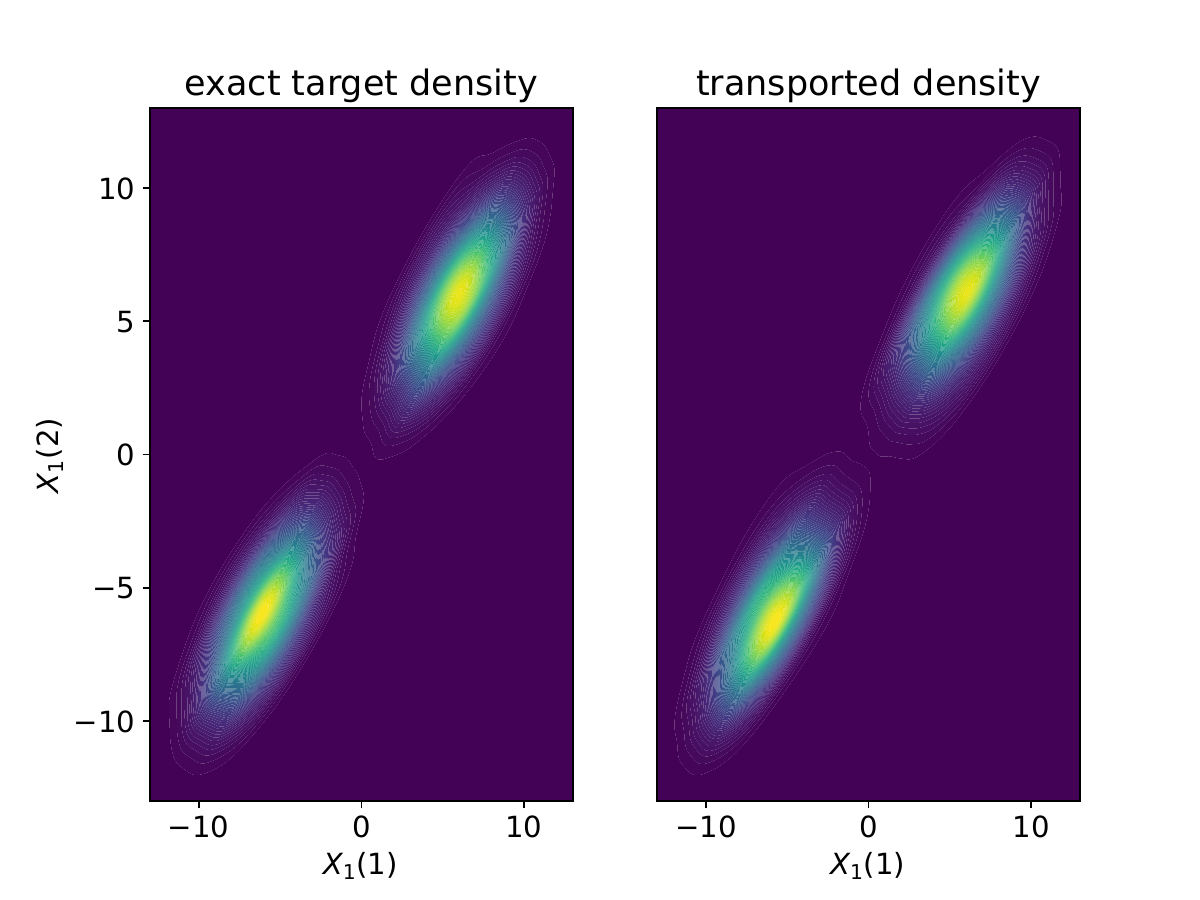}
         \hspace{1pt}
         \includegraphics[width=0.32\textwidth,trim={5 5 55 35},clip]{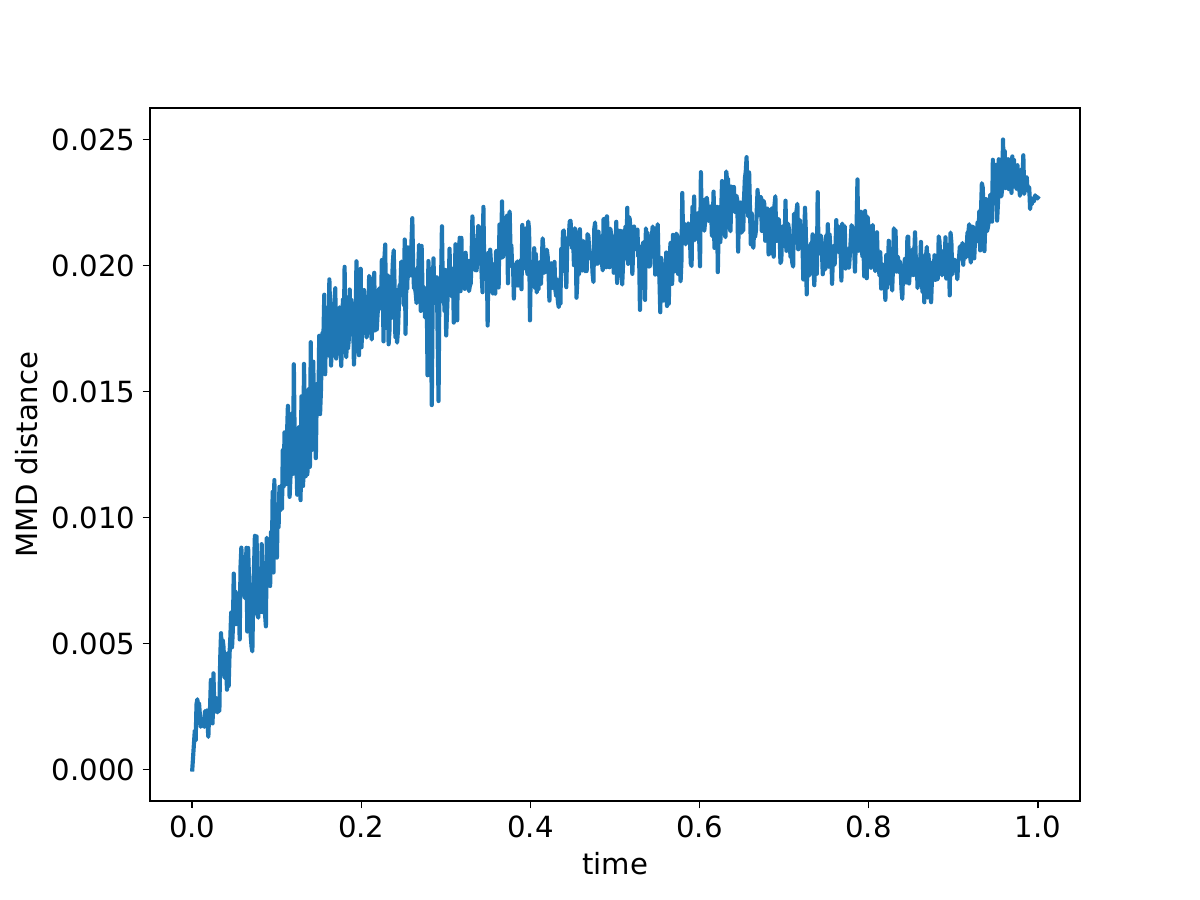}
         
         \includegraphics[width=0.32\textwidth,trim={5 5 55 35},clip]{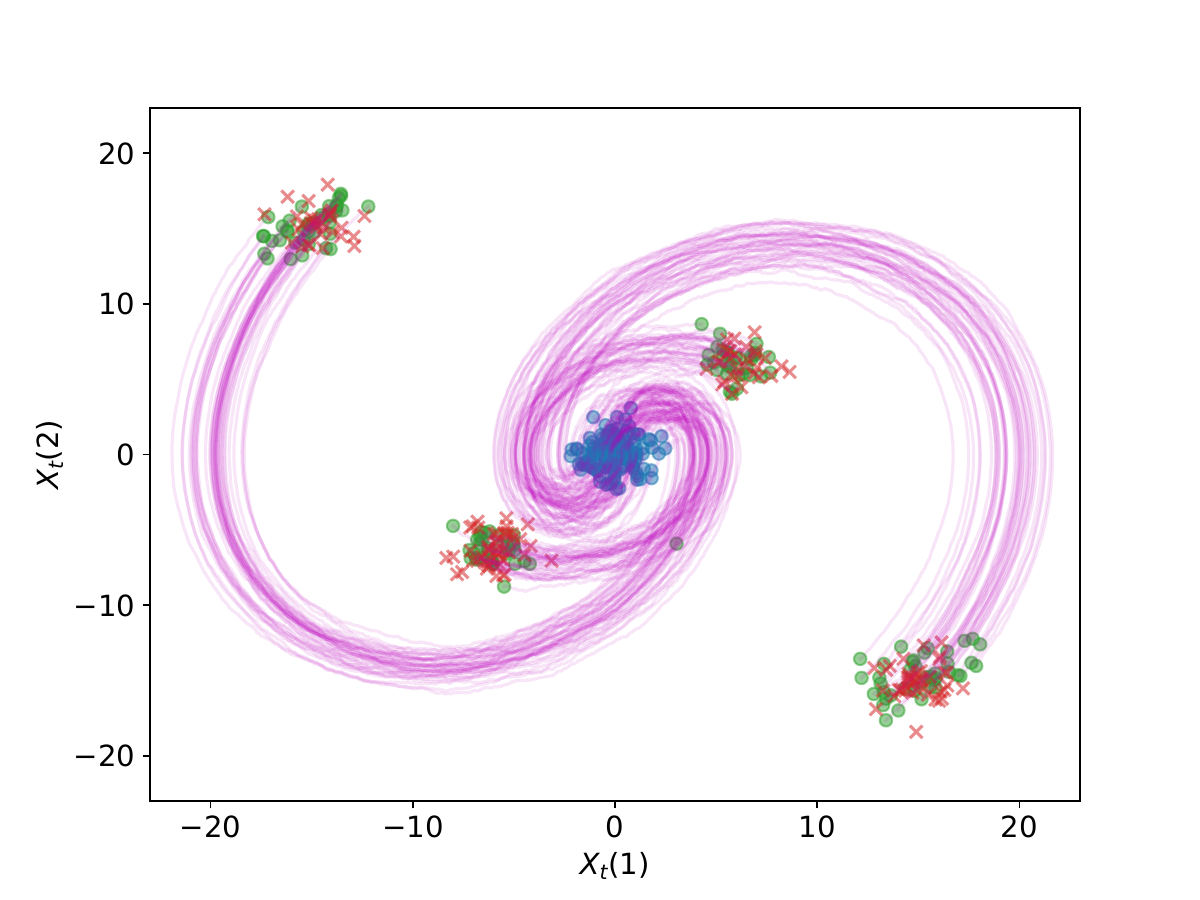}
         \hspace{1pt}
         \includegraphics[width=0.32\textwidth,trim={5 5 55 35},clip]{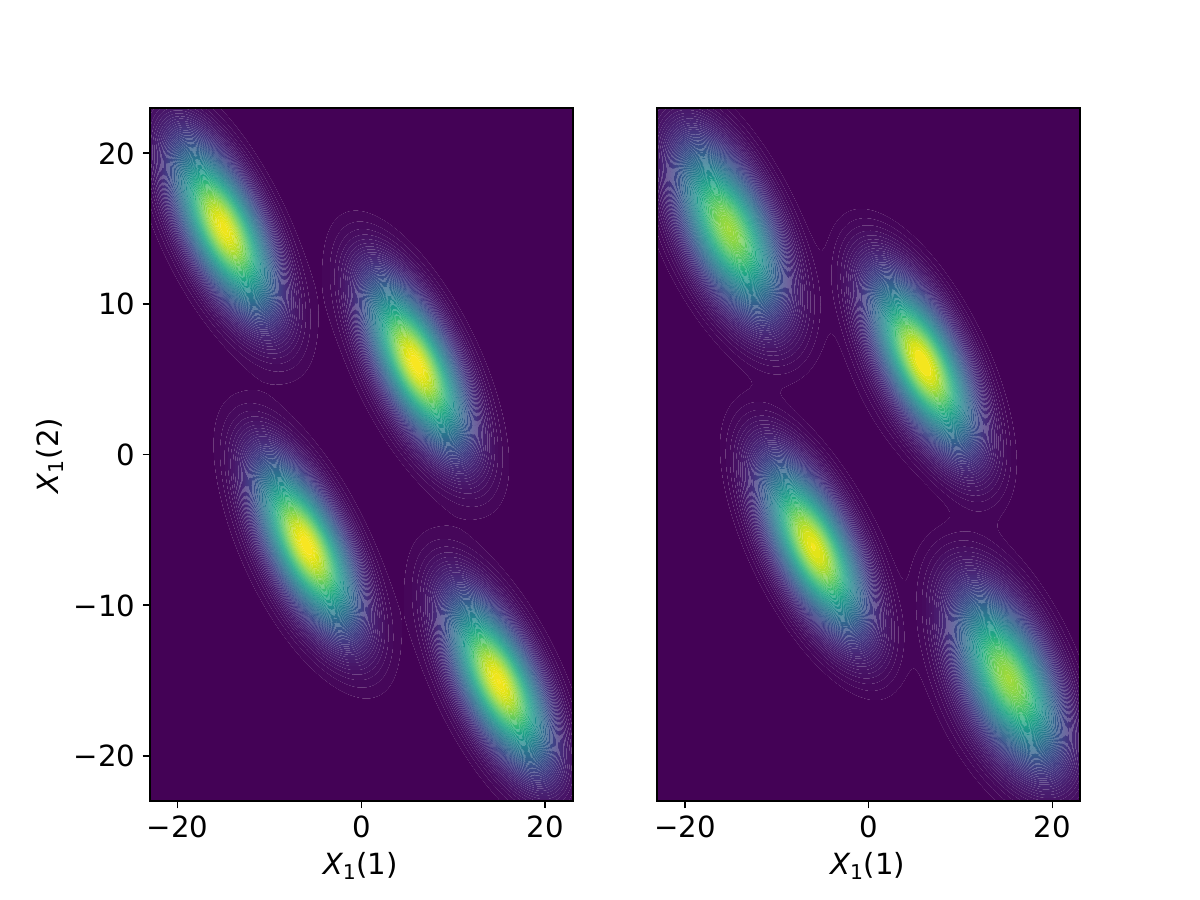}
         \hspace{1pt}
         \includegraphics[width=0.32\textwidth,trim={5 5 55 35},clip]{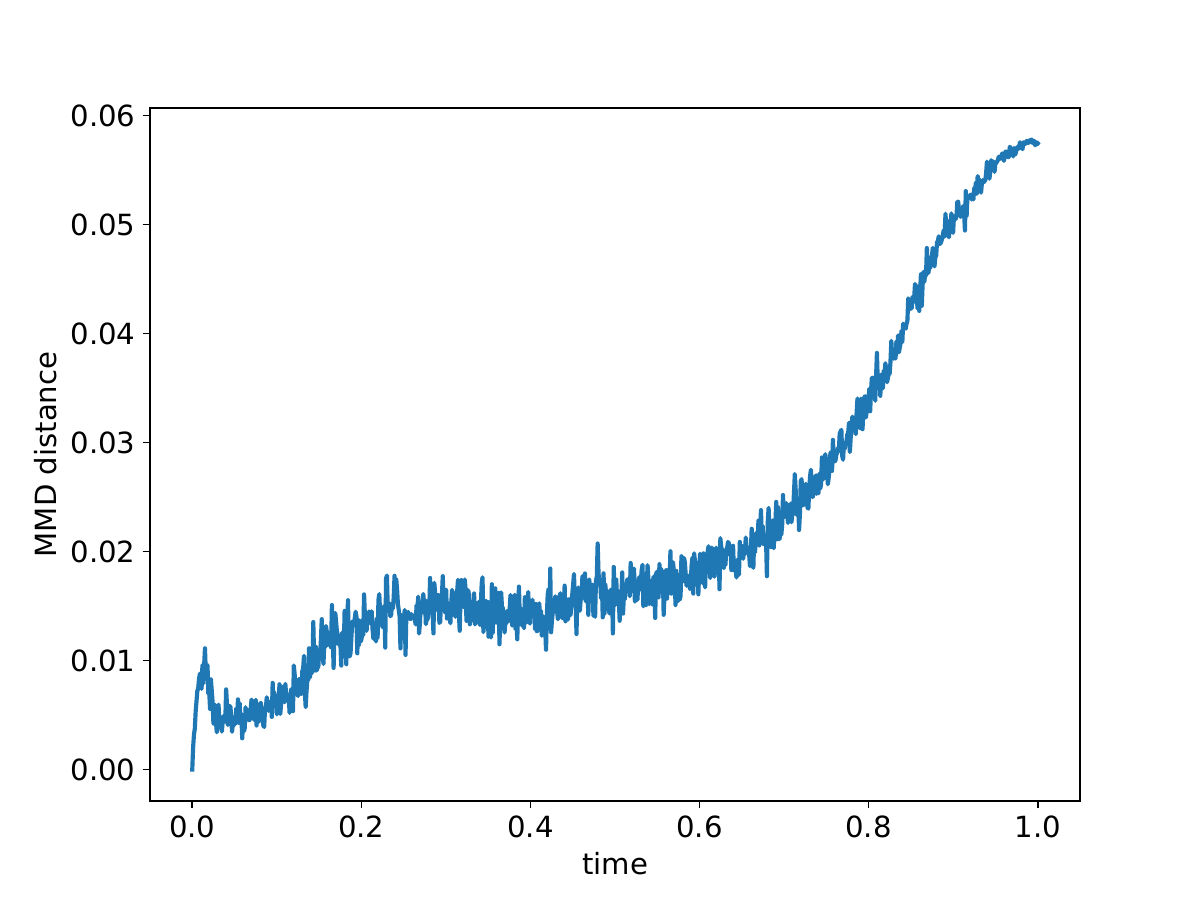}

          \includegraphics[width=0.32\textwidth,trim={5 5 55 35},clip]{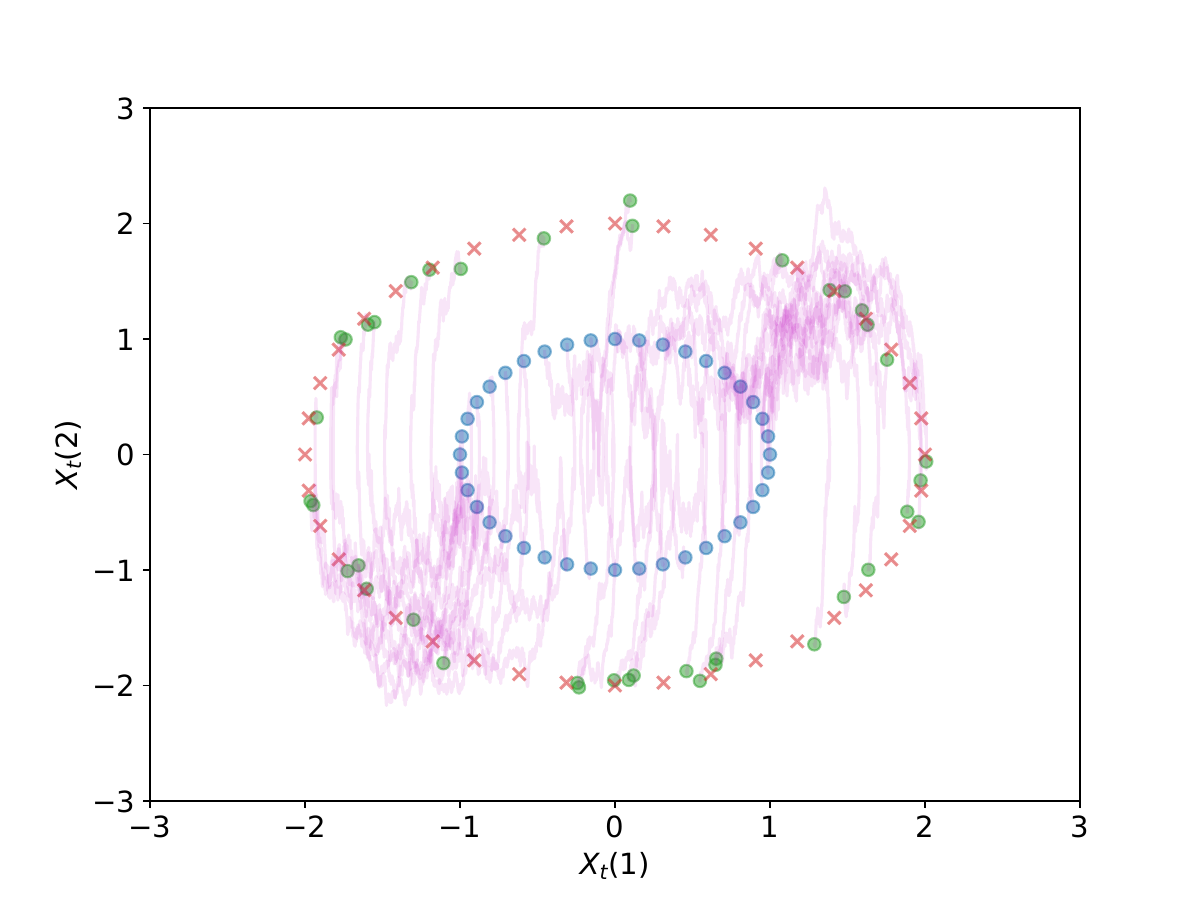}
         \hspace{1pt}
         \includegraphics[width=0.32\textwidth,trim={5 5 55 35},clip]{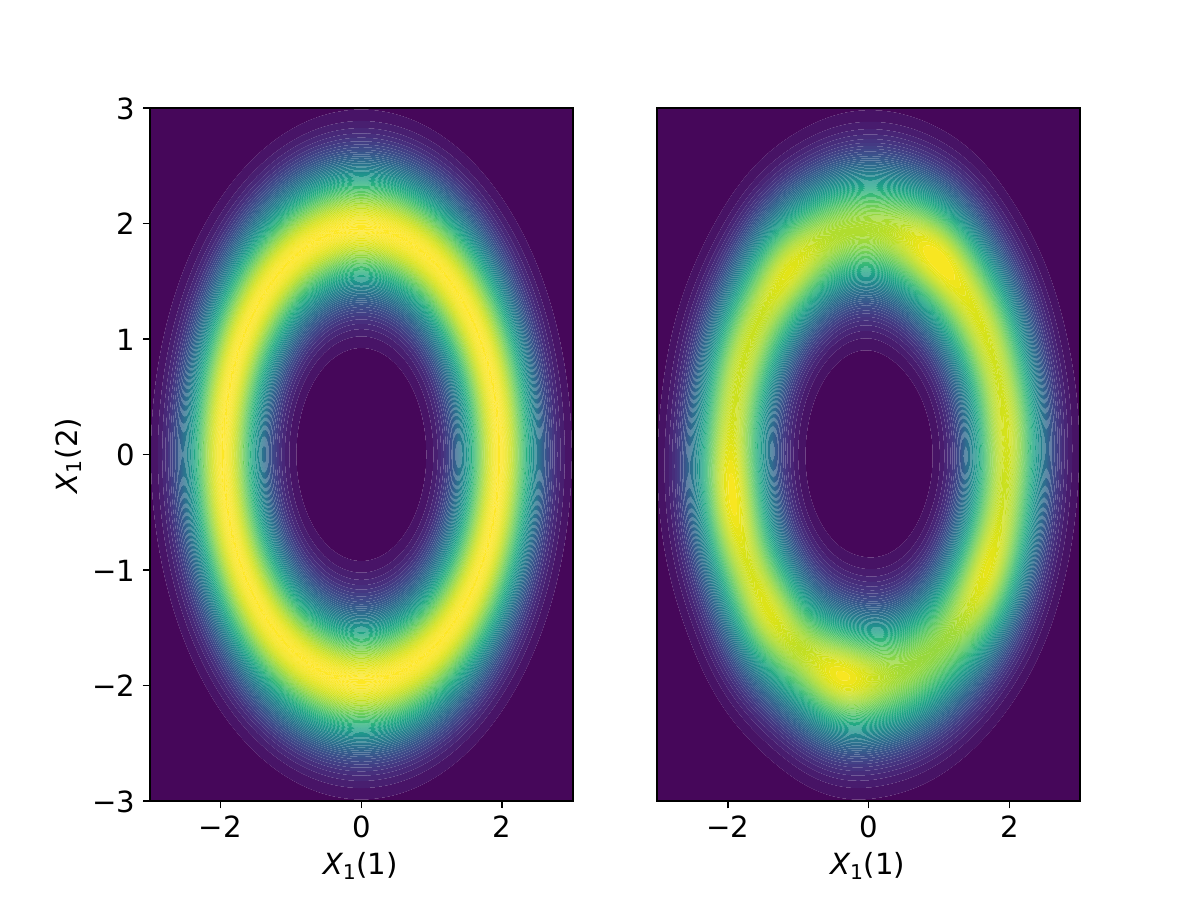}
         \hspace{1pt}
         \includegraphics[width=0.32\textwidth,trim={5 5 55 35},clip]{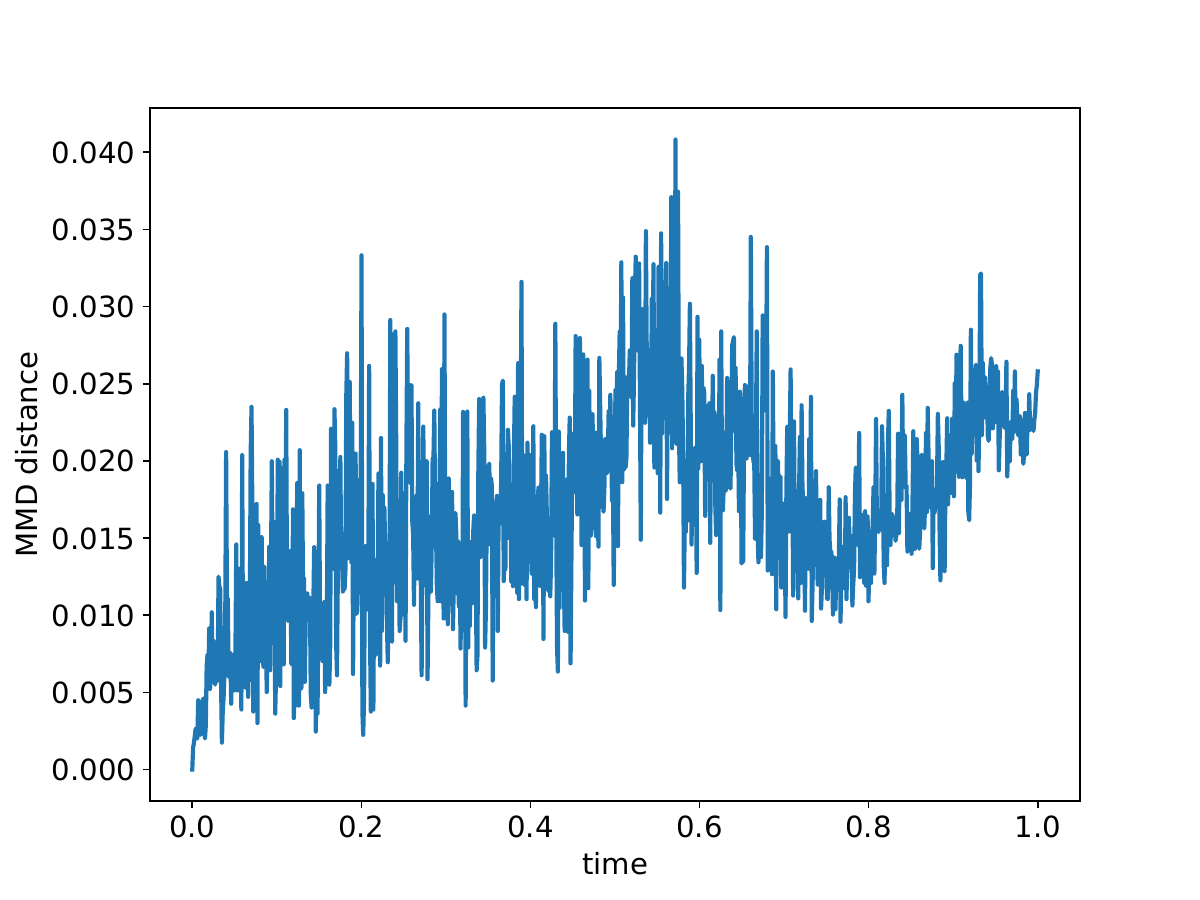}
    \caption{Numerical results for three 2-dimensional linear control systems, with model parameters~\eqref{eq:2d_systems}-(a)-(b)-(c). First row: numerical result for~\eqref{eq:2d_systems}-(a). The left panel shows the trajectories $\{X^i_t\}_{i=1}^{N'}$ generated from the prediction stage of the algorithm. The second panel compares the exact target density with a kernel density approximation of the generated samples  $\{X^i_1\}_{i=1}^{N'}$. The right panel shows the MMD distance between generated samples $\{X^i_t\}_{i=1}^{N'}$ and training samples $\{X^{z^i}_t\}_{i=1}^{N}$.  The second and third rows show similar results for models~\eqref{eq:2d_systems}-(b) and (c), respectively.}
    \label{fig:2d examples}
\end{figure}

\subsection{Numerical results for 2d systems}
The numerical results for three different 2-dimensional linear control systems, with model parameters~\eqref{eq:2d_systems}-(a)-(b)-(c), are presented in Figure~\ref{fig:2d examples}.  In all these cases, the coefficient $\epsilon=1$. 
The first row shows the results for~\eqref{eq:2d_systems}-(a) where the initial distribution is standard Gaussian and the target distribution is a mixture of two Gaussians. The left panel shows the trajectories $\{X^i_t;0\leq t\leq 1\}$ generated from the prediction stage. 
The initial states $\{X^i_0\}_{i=1}^{N'}$ are samples from the initial distribution and depicted by blue circles. 
The terminal states  $\{X^i_1\}_{i=1}^{N'}$ are depicted with green circles and are expected to represent samples from the target distribution. For comparison, independent samples from target distribution are shown as red ``$\times$'' markers. A kernel density approximation of the terminal states  is compared  with the exact target density in the second panel. The result demonstrates a qualitative proof that the proposed algorithm solves Problem~\ref{problem:cont-dist} for this example. The right panel shows a quantitative comparison between the empirical probability distributions of the generated trajectories $\{X^i_t\}_{i=1}^{N'}$ and training samples $\{X^{z^i}_t\}_{i=1}^N$, using the maximum mean discrepancy (MMD) distance with Gaussian kernel and bandwidth $2$. The MMD distance is normalized by the MMD distance between the initial and target distribution. The result highlights the conclusion of Theorem~\ref{thm:control} that $X_t \overset{\text{d}}{=}X^z_t$ for all $t\in[0,1]$.  

The second and third rows show the same results for modeling parameters~\eqref{eq:2d_systems}-(b) and (c), respectively.  In the second row, the initial distribution is standard Gaussian, while the terminal distribution is a mixture of four Gaussians. In the third row, the initial and target distributions are uniform distributions on circles of different radius. The results serve as a proof of concept that the proposed algorithm solves Problem~\ref{problem:cont-dist} with a reasonable accuracy among different class of linear systems and probability distributions.

\subsection{Numerical results for higher dimensional systems}
We explore the scalability of the algorithm with the problem dimension by considering a mass-spring model with model parameters as in~\cite[Sec. IV-C]{mei2024time}. 
Figure~\ref{fig:high dimensional examples} shows the results for $4$-dimensional and $8$-dimensional mass-spring systems. The target distribution is a mixture of four Gaussians. The left panel shows the approximated and exact densities projected onto the last two components. The panel at the center shows the normalized MMD distance between the empirical distribution of the generated samples and training samples. The right panel shows the Wasserstein-2\footnote{The $W_2$ is computed using the python optimal transport (POT) library with the squared Euclidean distance.} $(W_2)$ distance between the same samples. The results highlight the ability of the algorithm to scale with the problem dimension.

\begin{figure}[t]
    \centering
          \includegraphics[width=0.32\textwidth,trim={5 5 55 35},clip]{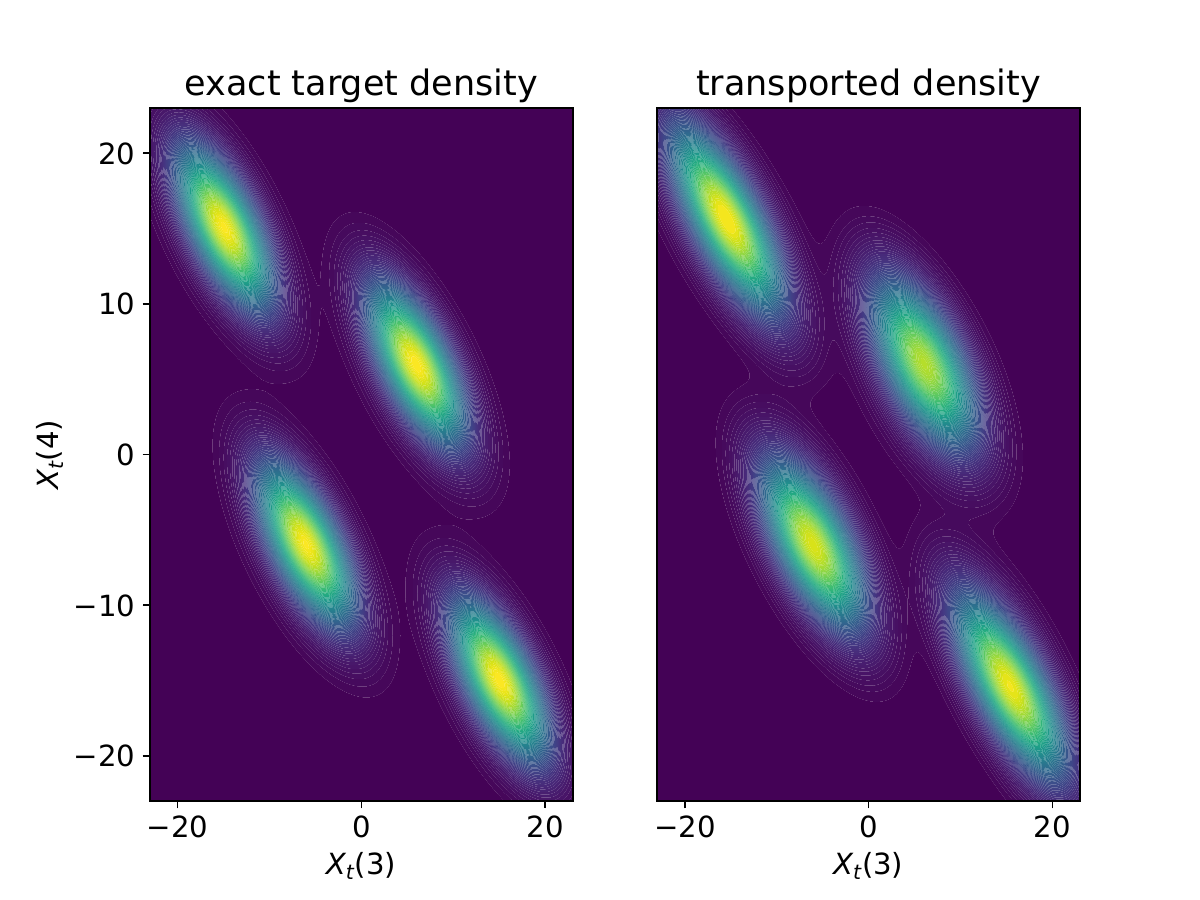}
          \hspace{1pt}
          \includegraphics[width=0.32\textwidth,trim={5 5 55 35},clip]{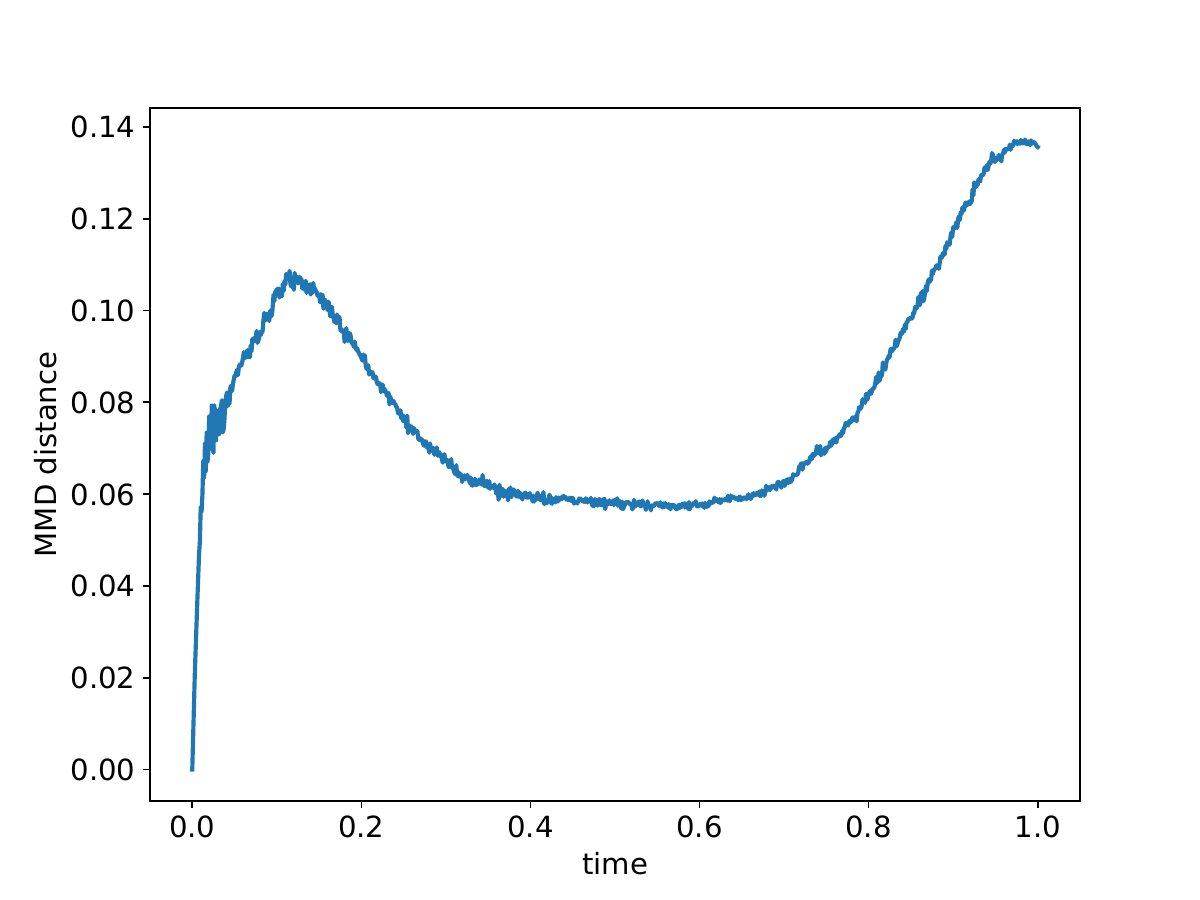}
          \hspace{1pt}
          \includegraphics[width=0.32\textwidth,trim={5 5 55 35},clip]{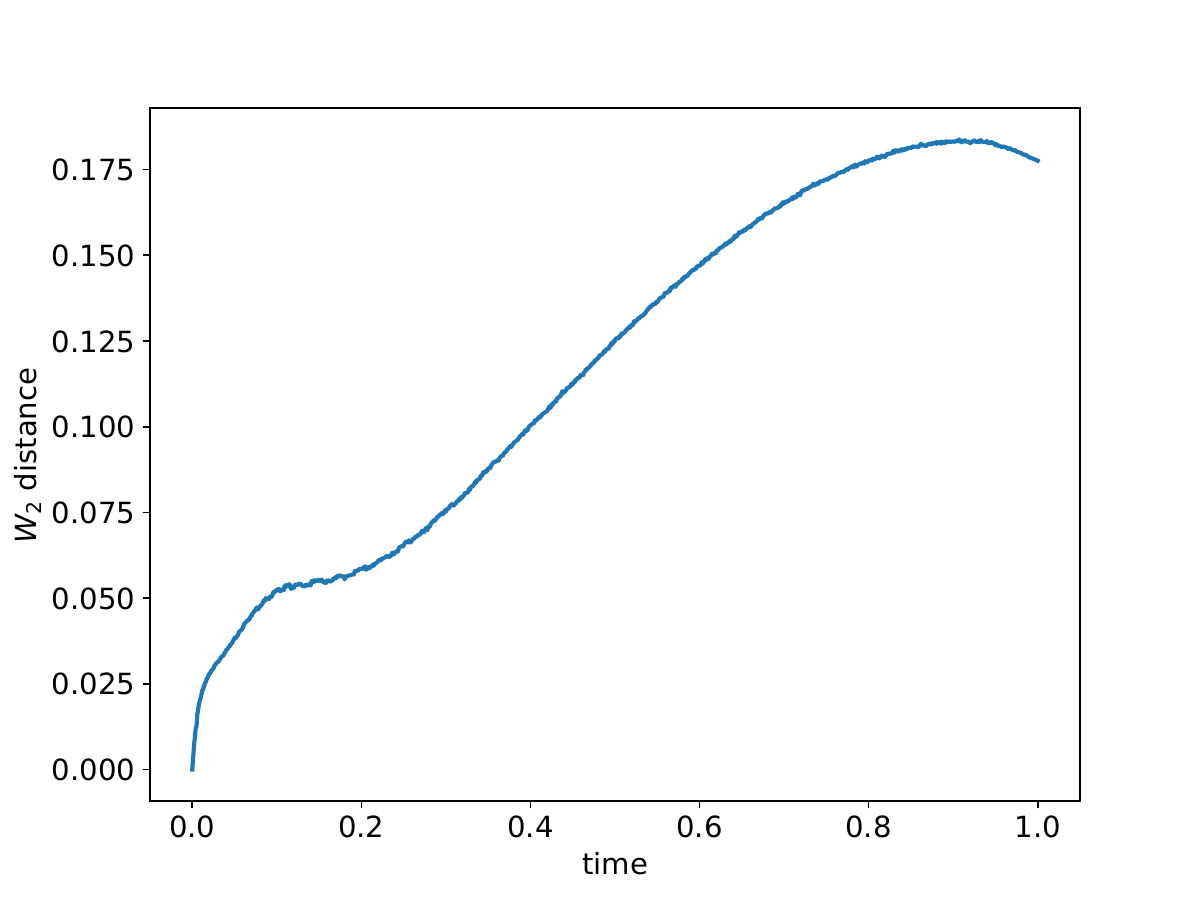}

          \includegraphics[width=0.32\textwidth,trim={5 5 55 35},clip]{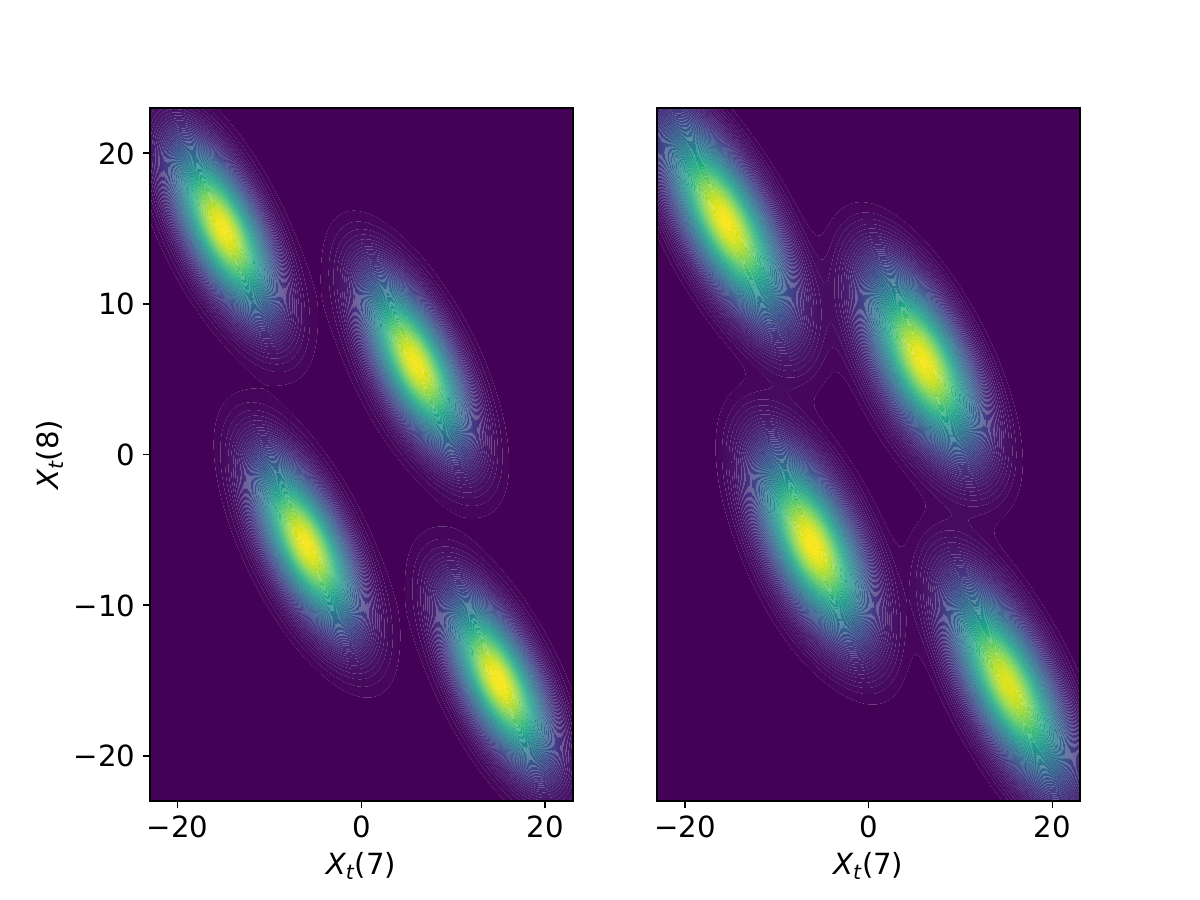}
          \hspace{1pt}
          \includegraphics[width=0.32\textwidth,trim={5 5 55 35},clip]{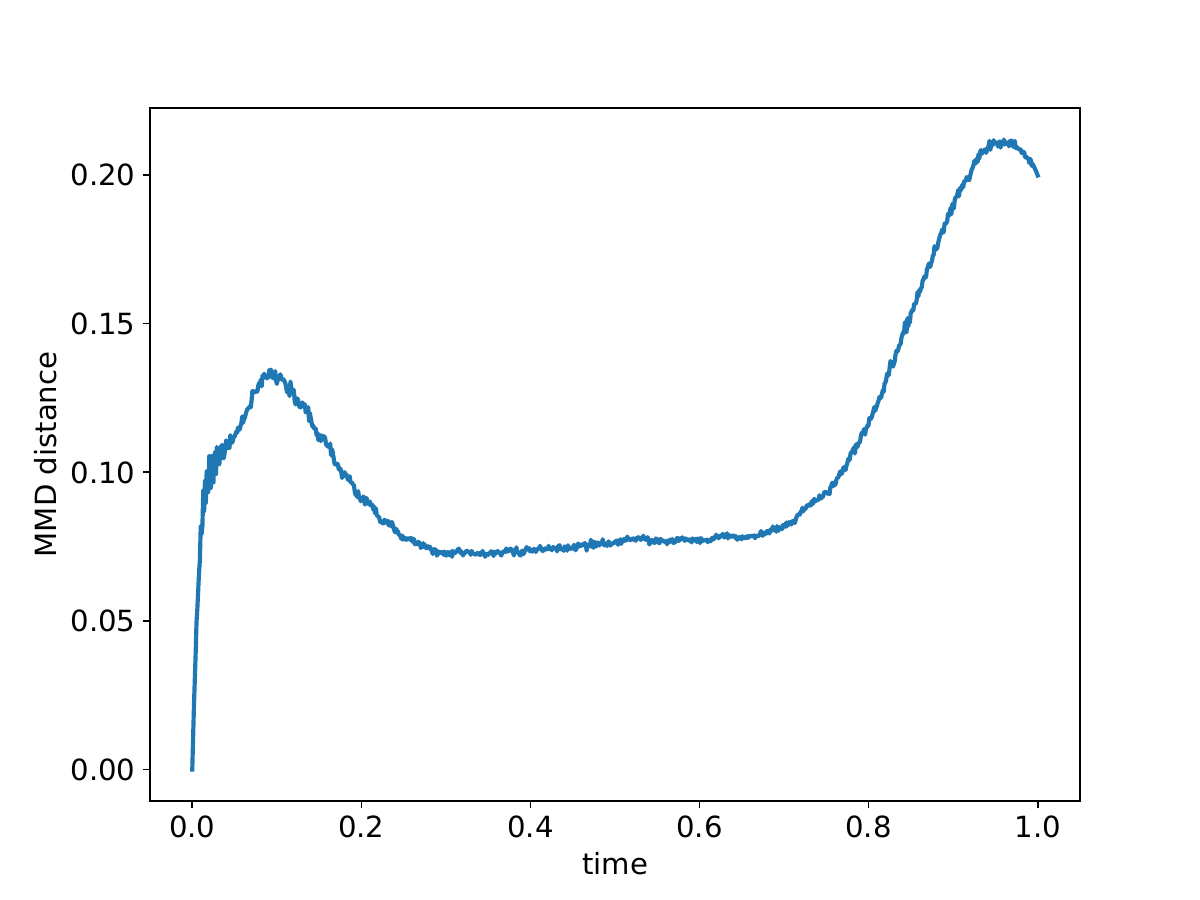}
          \hspace{1pt}
          \includegraphics[width=0.32\textwidth,trim={5 5 55 35},clip]{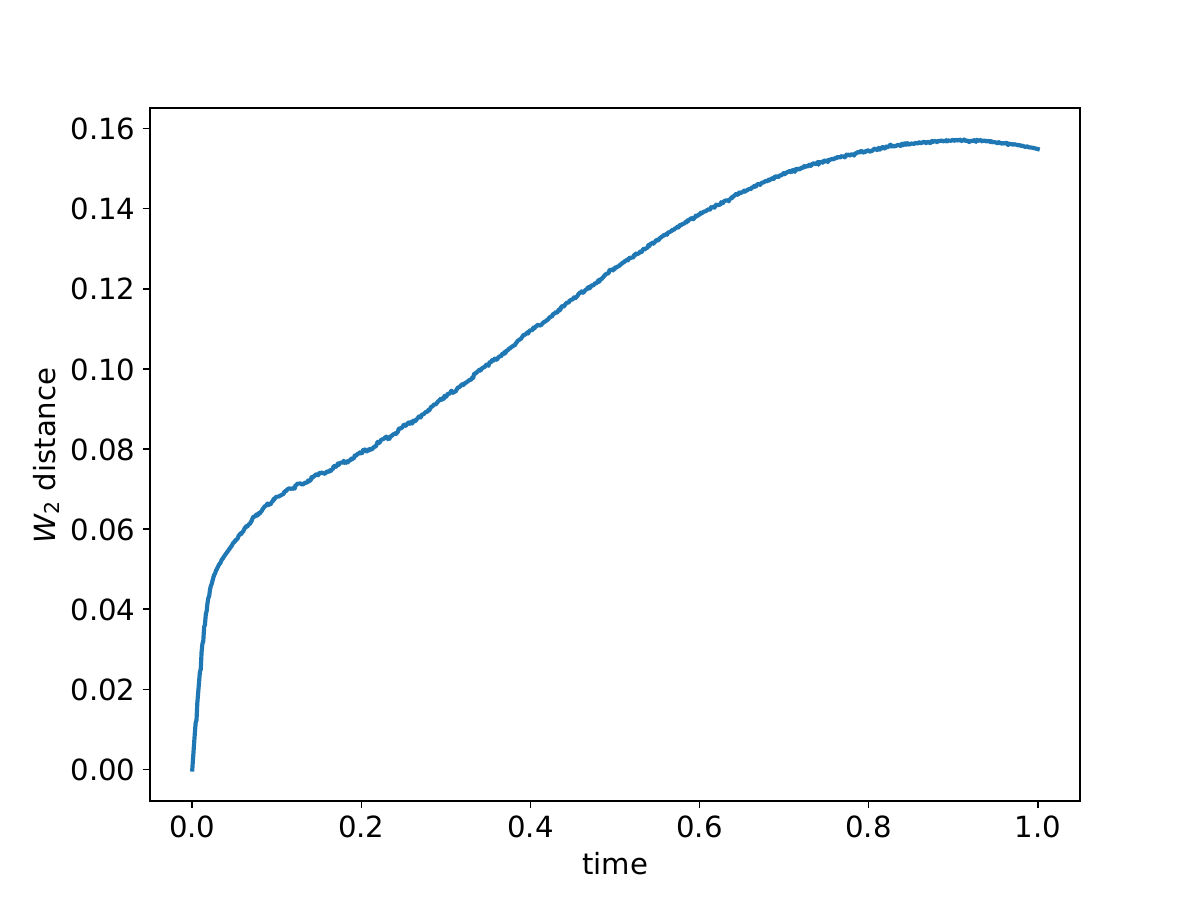}
 
    \caption{Numerical results for four dimensional (first row) and eight dimensional (second row) mass-springs systems. The left panel shows the approximated and exact densities, at $t=1$, projected onto the last two components. The panel at the center shows the MMD distance, and the right panel shows the W2 distance between the generated samples and training samples as a function of time t.}
    \label{fig:high dimensional examples}
\end{figure}
\section{Conclusion}
In this work, we introduced a novel framework for flow matching within the context of deterministic and stochastic linear control systems. By leveraging the structure of these systems, we developed an efficient methodology to steer an initial probability distribution to a desired target distribution while adhering to control constraints. The numerical experiments demonstrated the effectiveness and scalability of our approach across various system parameters and distribution classes. Future research could explore performance enhancements and extensions to control-affine systems or systems with state constraints,  to further broaden the applicability of this methodology.


\newpage
 \acks{The authors Y. Mei, M. Al-Jarrah, and A. Taghvaei are supported by the
National Science Foundation (NSF) award 2318977 and 2347358. The author Y. Chen is supported by the NSF awards 1942523 and 2008513. The authors thank Bamdad Hosseini for  his  valuable insights.}

\bibliography{references,references-thermo}

\end{document}